\documentclass[11pt]{amsart}
\usepackage{amsmath, amssymb, amsthm, amsfonts}
\usepackage{bbm}

\usepackage{ifthen}
\usepackage{tikz}
\usetikzlibrary{positioning,decorations.pathreplacing}

\usepackage{cite}
\usepackage{appendix}
\usepackage{graphicx}
\usepackage{color}
\usepackage{color}
\usepackage{algorithm}
\usepackage{algorithm}
\usepackage[noend]{algpseudocode}
\usepackage{epstopdf}
\usepackage{wrapfig}
\usepackage{paralist}
\usepackage{wasysym}
\usepackage{indentfirst}
\usepackage[textsize=tiny]{todonotes}

\usepackage{framed}
\usepackage[framemethod=tikz]{mdframed}
\usepackage[bottom]{footmisc}
\usepackage{enumitem}
\setitemize{noitemsep,topsep=3pt,parsep=3pt,partopsep=3pt}
\usepackage[font=small]{caption}
\usepackage{xspace}

\newtheorem{theorem}{Theorem}[section]
\newtheorem{lemma}[theorem]{Lemma}
\newtheorem{meta-theorem}[theorem]{Meta-Theorem}

\newtheorem{corollary}[theorem]{Corollary}

\definecolor{darkgreen}{rgb}{0,0.5,0}
\usepackage{hyperref}
\hypersetup{
    unicode=false,          
    colorlinks=true,        
    linkcolor=red,          
    citecolor=darkgreen,        
    filecolor=magenta,      
    urlcolor=cyan           
}
\usepackage[capitalize, nameinlink]{cleveref}
\Crefname{lemma}{Lemma}{Lemmas}
\Crefname{claim}{Claim}{Claims}
\Crefname{remark}{Remark}{Remarks}
\Crefname{observation}{Observation}{Observations}
\algnewcommand\algorithmicswitch{\textbf{switch}}
\algnewcommand\algorithmiccase{\textbf{case}}

\algdef{SE}[SWITCH]{Switch}{EndSwitch}[1]{\algorithmicswitch\ #1\ \algorithmicdo}{\algorithmicend\ \algorithmicswitch}%
\algdef{SE}[CASE]{Case}{EndCase}[1]{\algorithmiccase\ #1}{\algorithmicend\ \algorithmiccase}%
\algtext*{EndSwitch}%
\algtext*{EndCase}%

\renewcommand{\paragraph}[1]{\vspace{0.15cm}\noindent {\bf #1.}}

\newcommand{\FullOrShort}{full}

\ifthenelse{\equal{\FullOrShort}{full}}{

  \newcommand{\fullOnly}[1]{#1}
  \newcommand{\shortOnly}[1]{}
}{
    \newcommand{\shortOnly}[1]{#1}
    \newcommand{\fullOnly}[1]{}
  }

\usepackage{amsaddr}
\date{}
\title[Color Degree Sequences in Gallai Colorings]
{Full Characterization of Color Degree Sequences in Complete Graphs Without Tricolored Triangles}
\author{Anton Trygub}
\address{Massachusetts Institute of Technology}
\email{trygub@mit.edu}

\begin{document}

\maketitle

\begin{abstract}
For an edge-colored complete graph, we define the color degree of a node
as the number of colors appearing on edges incident to it.
In this paper, we consider colorings that don't contain tricolored triangles
(also called rainbow triangles); these colorings are also called Gallai colorings.

We give a complete characterization of all possible color degree sequences
$d_1 \le d_2 \le \dots \le d_n$ that can arise on a Gallai coloring of $K_n$:
it is necessary and sufficient that
\[ \sum_{i = k}^n \frac{1}{2^{d_i - d_{k-1}}} \ge 1 \]
holds for all $1 \le k \le n$, where $d_0=0$ for convenience.
As a corollary, this gives another proof of a 2018 result of Fujita, Li, and Zhang
who showed that the minimum color degree in such a coloring is at most $\log_2n$.
\end{abstract}

\setcounter{page}{1}
\thispagestyle{empty}

\vspace{1cm}


\section{Introduction}
\subsection{Background}
In this paper, we consider edge colorings of complete graphs on $n$ nodes (later denoted $K_n$)
that don't contain tricolored triangles (also called rainbow triangles);
these colorings are also called \emph{Gallai colorings}.

The name ``Gallai colorings'' comes from the 1967 paper by Gallai,
who showed the following result.
\begin{theorem}
    [\cite{gallai}]
    \label{thm:gallai}
    In any rainbow triangle-free coloring of a complete graph,
    there exists a non-trivial partition of the vertices such that between the parts,
    there are a total of at most two colors, and,
    between every pair of parts, there is only one color on the edges.
\end{theorem}

A lot of properties of Gallai colorings have been discovered over the years. Erdős, Simonovits and Sós in \cite[Appendix A]{EST} showed that the maximum possible number of different colors in a Gallai coloring of $K_n$ is $n-1$. Gyárfás and Simonyi \cite{GS} showed that in every Gallai coloring there is a monochromatic spanning tree with height at most two, and that in any Gallai coloring of $K_n$ there exists a color with a maximum degree at least $\frac{2n}{5}$. The interest of this paper is in the color degrees of nodes in Gallai colorings.


\subsection{Complete characterization of color degrees}
Given a edge-coloring of a graph,
we define the \emph{color degree} $d(i)$ of vertex $i$
as the number of distinct colors that appear on edges incident to $i$.
In 2018, Fujita, Li, and Zhang proved an upper bound
on the minimum color degree that can appear in Gallai colorings.
\begin{theorem}
    [{\cite[Theorem 3]{FLZ}}]
    In any Gallai coloring of $K_n$ with vertex set $\{1,2,\dots,n\}$,
    there exists a vertex $i$ with $d(i) \le \log_2 n$.
    \label{thm:FLZ}
\end{theorem}

Construction  due to Li and Wang ({\cite[Proposition 2.6]{LW}}) shows that this bound is tight.

In this paper, we will prove a generalization of this result
which characterizes exactly what sequences $d(i)$ may arise from a
Gallai coloring.
The main result of our paper is the following pair of theorems:
\begin{theorem}
    Let $d(1) \le d(2) \le \dots \le d(n)$ be the color degree sequence
    of some Gallai coloring of graph $K_n$.
    Let's also define $d(0) = 0$.
    Then, for each $1 \le k \le n$, the following inequality holds:
    \[ \sum_{i = k}^n \frac{1}{2^{d(i) - d(k-1)}} \ge 1 \]
    \label{thm:mainA}
\end{theorem}

\begin{theorem}
    Conversely, suppose
    $d_0 = 0 < d_1 \le d_2 \le d_3 \le \dots \le d_n$ are some integers,
    which satisfy the following condition for every $1 \le k \le n$:
    \[ \sum_{i = k}^n \frac{1}{2^{d_i - d_{k-1}}} \ge 1 \]
    Then, there exists a Gallai coloring of $K_n$ with $d(i) = d_i$.
    \label{thm:mainB}
\end{theorem}

Notice in particular the $k=1$ case of Theorem~\ref{thm:mainA}
is independently of interest. It reads:
\begin{theorem}
    For any Gallai coloring of $K_n$, the following inequality holds:
    \[ \sum_{i = 1}^n \frac{1}{2^{d(i)}} \ge 1 \]
    \label{thm:mainC}
\end{theorem}
Our result then implies Theorem~\ref{thm:FLZ} directly.
(Indeed, if $x$ is the minimum color degree, then
\[ \frac{n}{2^x} \ge \sum_{i = 1}^n
\frac{1}{2^{d(i)}} \ge 1 \implies x \le \log_2{n} \]
so Theorem~\ref{thm:FLZ} follows.)

\subsection{Roadmap}
The rest of the paper is structured as follows.
After setting up a small amount of notation and tools in Section~\ref{sec:setup},
we will prove Theorem~\ref{thm:mainC} in Section~\ref{sec:proofC}.
Then we will use Theorem~\ref{thm:mainC} to prove
Theorem~\ref{thm:mainA} in Section~\ref{sec:proofA}.

\subsection{Acknowledgments} I want to thank Evan Chen for helpful discussions and advice with structuring the paper.  

\section{Setup}
\label{sec:setup}
We will denote the color of the edge between nodes $i, j$ by $c(i, j)$.

We will need the following simple corollary of Theorem~\ref{thm:gallai}, also present in {\cite[Lemma A]{GS}}.
\begin{corollary}
    If Gallai coloring contains at least three different colors,
    then there is a color that spans a disconnected graph.
    \label{cor:gallai}
\end{corollary}
\begin{proof}
    Suppose that there are at least three colors.
    Consider partition from Theorem~\ref{thm:gallai},
    and any color which isn't among two colors connecting nodes
    from different parts.
    The graph spanned by this color can't be connected.
\end{proof}

We will also prove the following simple lemma.
\begin{lemma}
    Suppose that color $1$ spans a disconnected graph,
    let $A$ be one of the connected components of this graph.
    Then for any nodes $u, v \in A$ and $w \notin A$, $c(u, w) = c(v, w)$.
    \label{lem:simple}
\end{lemma}

\begin{proof}
First, consider any edge $(u_1, v_1)$ of color $1$ in $A$.
For any node $w\notin A$, consider triangle $u_1v_1w$.
Then $c(u_1, v_1) = 1$, but $c(u_1, w) \neq 1, c(v_1, w) \neq 1$, so $c(u_1, w) = c(v_1, w)$.

Now, consider any nodes $u, v \in A$.
There is a simple path of color $1$ between then,
let it be $x_0, x_1, \dots, x_k$ with $x_0 = u, x_k = v$.
Then $c(w, u) = c(w, x_0) = c(w, x_1) = \dots = c(w, x_k) = c(w, v)$, as desired.
\end{proof}

\section{Proof of Theorem~\ref{thm:mainC}}
\label{sec:proofC}
We now prove that in any Gallai coloring,
we have \[ \sum_{i = 1}^n \frac{1}{2^{d(i)}} \ge 1. \]

\begin{proof}
    [Proof of Theorem~\ref{thm:mainC}]
    We will prove this statement by induction by $n$. It's trivial for $n\le 3$.

    Consider $n\ge 4$.
    If there are at most $2$ different colors,
    then $\sum_{i = 1}^n \frac{1}{2^{d(i)}} \ge \frac{n}{4} \ge 1$, as desired.
    Otherwise, there exists a color that spans a disconnected graph.
    Wlog, this is color $1$.

    Consider any connected component of color $1$, $A$.
    By Lemma~\ref{lem:simple}, for any node $v\notin A$
    there exists some color $c(v)$,
    such that all edges between $v$ and $A$ have color $c(v)$.

    Let $k$ be the number of different colors among $c(v)$ (for $v\notin A$).
    Consider the graph obtained by compressing $A$ into a single node $a$ with $d(a) = k$.
    Here, $c(a, v) = c(v)$ for $v \notin A$.
    By the induction hypothesis, we have:
    \[ \sum_{v \notin A} \frac{1}{2^{d(v)}} + \frac{1}{2^k} \ge 1 \]
    Now, consider separately the graph spanned by nodes in $A$.
    Let $d_1(v)$ denote the color degree of node $v$ with respect to this graph
    (where $v \in A$).
    By the induction hypothesis, we have:
    \[ \sum_{v \in A} \frac{1}{2^{d_1(v)}} \ge 1 \]
    Now note that for each $v\in A$, $d(v) \le d_1(v) + k$
    (as outside of $A$ $v$ has precisely $k$ different colors).
    Then we can write:
    \begin{align*}
    \sum_{i = 1}^n \frac{1}{2^{d(i)}}
    &= \sum_{v \notin A} \frac{1}{2^{d(v)}} + \sum_{v \in A} \frac{1}{2^{d(v)}} \\
    &\ge \sum_{v \notin A} \frac{1}{2^{d(v)}} + \sum_{v \in A} \frac{1}{2^{d_1(v) + k}} \\
    &= \sum_{v \notin A} \frac{1}{2^{d(v)}}
    + \frac{1}{2^k}\sum_{v \in A} \frac{1}{2^{d_1(v)}} \\
    &\ge \sum_{v \notin A} \frac{1}{2^{d(v)}} + \frac{1}{2^k} \\
    &\ge 1. \qedhere
    \end{align*}
\end{proof}

\section{Proof of Theorem~\ref{thm:mainA}}
\label{sec:proofA}
We first prove the following lemma, which may be of independent interest.

\begin{lemma}
    Consider Gallai coloring of $K_n$, in which one of color degrees is $n-1$.
    Then the color degrees are $n-1, n-1, n-2, n-3, \dots, 1$ in some order.
    \label{lem:maxcase}
\end{lemma}

\begin{proof}
    Without loss of generality, let node $n$ have $d(n) = n-1$.
    All colors $c(i, n)$ for $1 \le i \le n-1$ are different, wlog $c(i, n) = i$ for $1 \le i \le n-1$.

    For any $1 \le i < j \le n-1$, from triangle $ijn$ we get that $c(i, j)$ is $i$ or $j$. For each $i, j$, let's draw an oriented edge $i \to j$ if $c(i, j) = i$, and $j \to i$ if $c(i, j) = j$. As there are no rainbow triangles, there are no directed cycles of length $3$ in this graph.

    Remember that if the tournament graph contains no directed cycle of length $3$, it is acyclic. This is a well-known fact, and can be proved by considering the shorted directed cycle $(x_1, x_2, \dots, x_k, x_1)$. Clearly $k \ge 4$. If edge between $x_1, x_3$ is oriented $x_3 \to x_1$, then there is a shorter directed cycle $(x_1, x_2, x_3, x_1)$, else there is a shorter directed cycle $(x_3, x_4, \dots, x_1, x_3)$. So, nodes $1, 2, \dots, n-1$ can be arranged in a row, so that all edges go from left to right. The color degree of the $i$-th node in this row will be exactly $i$.
\end{proof}

\begin{proof}
    [Proof that Theorem~\ref{thm:mainC} implies Theorem~\ref{thm:mainA}]

    For $k = 1$ this result is the same as the result of
    Theorem~\ref{thm:mainC}. Let's consider $k \ge 2$.

    Consider any $i \ge k$. How many different colors go from $i$ to nodes $1, 2,
    \dots, k-1$? Suppose that there are edges from at least $x$ different colors
    from $i$ to nodes from $1$ to $k-1$. Let's choose one edge of each color, and
    denote the endpoints of these edges different from $i$ by $v_1, v_2, \dots,
    v_x$. Consider graph on nodes $v_1, v_2, \dots, v_x$ and node $i$. It contains
    $x+1$ nodes, and the color degree of node $i$ in it is $x$, so in this subgraph,
    at least one more node has to have color degree $x$, by Lemma~\ref{lem:maxcase}.
    But then $x \le d(k-1)$.

    Now, consider subgraph on nodes $k, k+1, \dots, n$. For any $i \ge k$, there
    are at most $d(k-1)$ colors going from $i$ to nodes $1, 2, \dots, k-1$, so the
    color degree of node $i$ in this new graph is at least $d(i) - d(k-1)$. Then we
    can just apply Theorem~\ref{thm:mainC} to this graph.
\end{proof}

\section{Proof of Theorem~\ref{thm:mainB}}
\label{sec:proofB}

Conversely, we prove that the condition from the Theorem~\ref{thm:mainA}
is also a sufficient condition for the existence of the Gallai coloring
with such a color degree sequence.

First, we will make a tool for constructing color degree sequences of Gallai
colorings of $K_{n+1}$ from color degree sequences of Gallai colorings of
$K_{n}$. For clarity, I will call valid color degree sequences of Gallai
colorings just \textit{valid sequences} (not necessarily sorted).

\begin{lemma}
    Suppose that $d_1, d_2, \dots, d_n$ is a valid sequence for $n$.
    Then, for any $v$, the following sequences are valid for $n+1$:
    \begin{itemize}
        \item $d_1, d_2, \dots, d_{v-1}, d_v, d_v, d_{v+1}, d_{v+2}, \dots, d_n$
        \item $d_1, d_2, \dots, d_{v-1}, d_v+1, d_v+1, d_{v+1}, d_{v+2}, \dots, d_n$
    \end{itemize}
    \label{lem:constructor}
\end{lemma}

\begin{proof}
    Consider any coloring of $K_n$ with $d(i) = d_i$ for all $i$. Let's
    duplicate node $v$: add a node $v_1$, with $c(v_1, i) = c(v, i)$ for $i \neq
    v$. Then, for all nodes except $v, v_1$, the color degree remains the same.

    Now, if we choose $c(v, v_1)$ equal to some color present among colors of
    edges incident to $v$, the color degree of $v$ won't change, else it will
    increase by $1$. These cases give two valid sequences for $n+1$.
\end{proof}

We can now prove Theorem~\ref{thm:mainB}.
\begin{proof}
    We will prove this theorem by induction by $n$.
    It clearly holds for $n = 2$, consider $n \ge 3$.

    Let $d_n = t$. From the condition for $k = n$ we get
    \[ \frac{1}{2^{t - d_{n-1}}} \ge 1, \]
    so $d_{n-1} = t$.
    We will show that at least one of the sequences
    $(d_1, d_2, \dots, d_{n-2}, t)$ and $(d_1, d_2, \dots, d_{n-2}, t-1)$ is valid for $n-1$,
    therefore the result will follow by Lemma~\ref{lem:constructor}.

    Let $m$ be the number of occurrences of $t$ in this color sequence,
    we know $m \ge 2$. We have $d_{n-m} < t$, and
    $d_{n-m+1} = d_{n-m+2} = \dots = d_n = t$.
    We will show that:

    \begin{itemize}
        \item If $m$ is even, the sequence
          \[ d_1, d_2, \dots, d_{n-m}, t-1, \underbrace{t, t, \dots,
          t}_\text{$m-2$ times} \]
        is valid for $n-1$.
        \item If $m$ is odd, the sequence
          \[ d_1, d_2, \dots, d_{n-m}, \underbrace{t, t, \dots, t}_\text{$m-1$
          times} \]
        is valid for $n-1$.
    \end{itemize}

    Let's first consider even $m$, and the sequence
    \[ d_1, d_2, \dots, d_{n-m}, t-1, \underbrace{t, t, \dots, t}_\text{$m-2$ times}. \]
    The condition is clearly true for $k \ge n-m+3$.
    If $m \ge 4$, it is also true for $k = n-m+2$
    (it's equivalent to $\frac{m-2}{2} \ge 1$).
    Consider $k \le n-m+1$. For such $k$,
    \begin{align*}
        \sum_{i = k}^{n-m} \frac{1}{2^{d_i - d_{k-1}}} + \frac{1}{2^{t-1 - d_{k-1}}} + \frac{m-2}{2^{t - d_{k-1}}}
        &= \sum_{i = k}^{n-m} \frac{1}{2^{d_i - d_{k-1}}} + \frac{m}{2^{t - d_{k-1}}} \\
        &= \sum_{i = k}^{n} \frac{1}{2^{d_i - d_{k-1}}} \\
        &\ge 1.
    \end{align*}

    Now, consider odd $m$, and sequence
    \[ d_1, d_2, \dots, d_{n-m}, \underbrace{t, t, \dots, t}_\text{$m-1$ times}. \]
    The condition is clearly true for $k \ge n-m+2$.
    Consider $k \le n-m+1$.
    We have to show that
    \[ \sum_{i = k}^{n-m} \frac{1}{2^{d_i - d_{k-1}}} + \frac{m-1}{2^{t -
    d_{k-1}}} = \sum_{i = k}^{n-1} \frac{1}{2^{d_i - d_{k-1}}} \ge 1 \]
    We will show that this follows
    from $\sum_{i = k}^{n} \frac{1}{2^{d_i - d_{k-1}}} \ge 1$ and $m$ being odd.
    Multiply this inequality by $2^{t - d_{k-1}}$ to get:
    \[ \sum_{i = k}^n 2^{t - d_i} \ge 2^{t - d_{k-1}} \]
    Note that here right part is even (as $d_{k-1}<t$),
    and the left part is odd (as $2^0$ appears there $m$ times).
    So, we can deduce the following inequality:
    \begin{align*}
        \sum_{i = k}^n 2^{t - d_i} - 1 &\ge 2^{t - d_{k-1}} \\
        \iff  \sum_{i = k}^{n-1} 2^{t - d_i} &\ge 2^{t - d_{k-1}} \\
        \iff \sum_{i = k}^{n-1} \frac{1}{2^{d_i - d_{k-1}}} &\ge 1
    \end{align*}
    This finishes the proof.
\end{proof}


\bibliographystyle{alpha}
\bibliography{refs}

\end{document}